\documentclass[final]{amsart}
\usepackage{verbatim}

\usepackage{mathtools}
\mathtoolsset{showonlyrefs}

\newtheorem{theorem}{Theorem}[section]
\newtheorem{lemma}[theorem]{Lemma}
\newtheorem{proposition}[theorem]{Proposition}
\newtheorem{definition}[theorem]{Definition}
\newtheorem{corollary}[theorem]{Corollary}

\def \E {\mathbb E}
\def \P {\mathbb P}
\def \R {\mathbb R}
\def \B {\mathbb B_n}
\def \BP {{\mathcal{BP}}_k^n}
\newcommand {\vol}[2][n]{\left|#2 \right|_{#1}}

\def \dovr {d_{\operatorname{ovr}}}

\title[The lower dimensional slicing inequality]{The lower dimensional slicing inequality for functions and related distance inequalities}

\author[J. Haddad]{Juli\'an Haddad}
\address{Departamento de An\'alisis Matem\'atico\\ Universidad de Sevilla\\ Sevilla, Espa\~na}
\email{{\tt jhaddad@us.es}}
\date{\today}

\begin{document}
\begin{abstract}
	It was shown in \cite{koldobsky2015slicing} that for every origin-symmetric star body $K \subseteq \R^n$ of volume $1$, every even continuous probability density $f$ on $K$ and $1 \leq k \leq n-1$, there exists a subspace $F \subseteq \R^n$ of codimension $k$ such that
	\[
		\int_{K \cap F} f \geq c^k (d_{\rm ovr}(K, \mathcal{BP}_k^n))^{-k}
	\]
	where $d_{\rm ovr}(K, \mathcal{BP}_k^n)$ is the outer volume ratio distance from $K$ to the class of generalized $k$-intersection bodies, and $c>0$ is a universal constant.
	The upper bound $d_{\rm ovr}(K, \mathcal{BP}_k^n) \leq c' \sqrt{n/k} \left(\log\left(\frac{en}k\right)\right)^{3/2}$ was established in \cite{koldobsky2011isomorphic} for every origin-symmetric convex body $K$. 
	In this note we show that there exist an origin-symmetric convex body $K$ of volume $1$ and an even continuous probability density $f$ supported on $K$ such that for every subspace $F$ of codimension $k$,
	\[
		\int_{K \cap F} f \leq \left( c \sqrt{\frac n{k \log(n)}  } \right)^{-k}. 
	\]
	As a consequence we obtain a lower bound for $d_{\rm ovr}(K, \mathcal{BP}_k^n)$ with $K$ a convex body, complementing the upper bound in \cite{koldobsky2011isomorphic}. This is
\[c \sqrt{n/k} (\log(n))^{-1/2} \leq \sup_K \dovr(K, \BP) \leq c' \sqrt{n/k} \left(\log\left(\frac{en}k\right)\right)^{3/2}.\]
	The case $k=1$ was obtained previously in \cite{klartag2018example,Klartag2020}.
\end{abstract}
\maketitle
\section{Introduction}
\label{sec_intro}

The problem known as the {\it slicing inequality for general measures} asks which is the greatest possible value $s_n$ such that for every origin-symmetric convex body $K$ (convex compact with non-empty interior) and every even continuous probability density $f$ on $K$, there exists a hyperplane $F \subseteq \R^n$ with
\[\int_{K \cap F} f \geq s_n.\]
Koldobsky investigated this problem in \cite{koldobsky2012hyperplane, koldobsky2014estimate, koldobsky2015slicing} and showed that $s_n \geq \frac 1{2\sqrt n}$, although the constant can be improved if $K$ is restricted to certain classes of convex bodies.

A {\it star body} is a set $K \subseteq \R^n$ of the form
\[K = \{0\} \cup \{x \in \R^n \setminus \{0\}: |x| \leq \rho_K(x/|x|)\} \]
where the {\it radial function} $\rho_K :S^{n-1} \to (0,\infty)$ is continuous and strictly positive, and $|\cdot|$ is the Euclidean norm.
Every convex body containing the origin in the interior is also a star body.
The radial distance between two star bodies $K,L$ is $\max_{v \in S^{n-1}} |\rho_K(v) - \rho_L(v)|$.

Given a star body $K \subseteq \R^n$, the {\it intersection body} of $K$ is the star body ${\rm I} K$ defined by $\rho_{{\rm I} K}(v) = \vol[n-1]{K \cap v^\perp}$, where $|\cdot|_k$ denotes the $k$-dimensional volume.
Taking the closure in the radial metric of the sets ${\rm I} K$ with $K \subseteq \R^n$ a star body, we get the class of {\it intersection bodies}, which we denote by $\mathcal I_n$.

For any class of star bodies $\mathcal C$ the outer volume ratio distance from a star body $K$ to $\mathcal C$ is defined by
\[\dovr(K, \mathcal C) = \inf \left\{ \left( \frac{\vol{D}}{\vol{K}} \right)^{1/n} : K \subseteq D, D \in \mathcal C \right\}. \]

In \cite{koldobsky2015slicing} the following inequality was proven:
\begin{theorem}
	\label{res_slicing_1}
	Let $K$ be an origin-symmetric star body in $\R^n$. Then for any even continuous non-negative function $f$ we have
	\[\int_{K} f \leq 2 d_{\rm ovr}(K, \mathcal I_n) \vol{K}^{1/n} \max_{v \in S^{n-1}} \int_{K \cap v^\perp} f.\]

	In particular, if $K$ is an origin-symmetric star body in $\R^n$ of volume $1$ and $f$ is an even continuous probability density in $K$ then
	\begin{equation}
		\max_{v \in S^{n-1}} \int_{K \cap v^\perp} f \geq (2 d_{\rm ovr}(K, \mathcal I_n))^{-1} .
	\end{equation}
\end{theorem}
Here $\dovr(K, \mathcal I_n)$ is the outer volume ratio distance between $K$ and the class of intersection bodies.
Since origin-symmetric ellipsoids are intersection bodies, if $K$ is an origin-symmetric convex body, John's theorem implies that $\dovr(K, \mathcal I_n) \leq \sqrt{n}$, hence the bound $s_n \geq \frac 1{2 \sqrt{n}}$.

The upper bound is $s_n \leq c/\sqrt{n}$, in other words, there exists an origin-symmetric convex body $K$ of volume $1$ and an even continuous probability density $f$ on $K$, such that
\[\max_{v\in S^{n-1}} \int_{K \cap v^\perp} f  \leq c \frac 1{\sqrt n}.\]
This estimate was obtained in \cite{klartag2018example} up to a factor $\sqrt{\log(\log(n))}$, and this factor was later removed in \cite{Klartag2020}.
Moreover, the estimates in \cite{klartag2018example, Klartag2020} also apply to non-central sections of $K$, this is,
\[\max_{v\in S^{n-1}, t \in \R} \int_{K \cap (v^\perp + tv)} f  \leq c \frac 1{\sqrt n}.\]
It follows from Theorem \ref{res_slicing_1} that the largest possible distance from a convex body to $\mathcal I_n$ has order $c \sqrt{n}$, this is,
\[c \sqrt n \leq \sup_K \dovr(K, \mathcal I_n) \leq \sqrt n \]
where the supremum runs over all origin-symmetric convex bodies $K \subseteq \R^n$, and $c>0$ is a universal constant.
A similar argument was used to show in \cite{bobkov2018estimates} (lower bound) and \cite{koldobsky2019measure} (upper bound) that
\[ c \sqrt{n/p} \leq \sup_K \dovr(K, L_p^n) \leq c' \sqrt{n/p} \]
for $p \geq 1$, where $L_p$ is the class of convex bodies $L$ such that $(\R^n, \|\cdot\|_L)$ embeds in $L_p$ as a Banach space (see \cite{koldobsky2005fourier} for details).
Also, in \cite{haddad2023radon} the convex body and measure constructed in \cite{klartag2018example} were adapted to a different slicing inequality from \cite{gregory2021inequalities} to obtain the lower bound
\begin{equation}
	\label{eq_estimate_Lp_lower}
	\sup_K \dovr(K, L_{-p}^n) \geq c n^{\frac 1{2p}},
\end{equation}
for $p \in [1,n)$, where the supremum runs over all origin-symmetric convex bodies $K \subseteq \R^n$, $c>0$ is a universal constant and $L_{-p}^n$ is the class of star bodies $L$ for which $(\R^n, \|\cdot\|_L)$ embeds in $L_{-p}$.
For integer $k \in [1,n)$, $L_{-k}$ is the class of $k$-intersection bodies (see \cite{koldobsky2005fourier} for details).

A version of Theorem \ref{res_slicing_1} with hyperplanes replaced by lower dimensional subspaces (Theorem \ref{res_slicing_k} below) was established also in \cite{koldobsky2015slicing}.
Denote by $G_{n,n-k}$ the Grassmanian of subspaces of $\R^n$ of codimension $k$.
The class of convex bodies responsible for the $k$-codimensional problem is the class of {\it generalized $k$-intersection bodies}, denoted by $\mathcal{BP}_k^n$, which can be defined as follows:
Given two star bodies $K,L$ and $p>0$, their radial $p$-sum is the star body $K \tilde+_p L$ whose radial function is 
\[\rho_{K \tilde+_p L}^p = \rho_K^p + \rho_L^p.\]
For $k \in [1,n)$, the class $\mathcal{BP}_k^n$ is defined as the closure in the radial metric of finite radial $k$-sums of origin-symmetric ellipsoids.
The class $\mathcal{BP}_k^n$ was introduced by Zhang in \cite{zhang1996sections} in connection with the lower dimensional Busemann-Petty problem. The lower-dimensional slicing inequality is as follows:

\begin{theorem}[{\cite[Corollary 1]{koldobsky2015slicing}}]
	\label{res_slicing_k}
	Let $K$ be an origin-symmetric star body in $\R^n$ and $1 \leq k \leq n-1$. Then for any even continuous non-negative function $f$,
	\begin{equation}
		\label{eq_slicing_k}
		\int_{K} f  \leq c^k d_{\rm ovr}(K, \mathcal{BP}_k^n)^k \vol{K}^{k/n} \max_{F \in G_{n,n-k}} \int_{K \cap F} f,
	\end{equation}
	where $c>0$ is a universal constant.
	In particular, if $K$ is an origin-symmetric star body in $\R^n$ of volume $1$ and $f$ is an even continuous probability density in $K$ then
	\begin{equation}
		\max_{F \in G_{n,n-k}} \int_{K \cap F} f \geq c^{-k} (d_{\rm ovr}(K, \mathcal{BP}_k^n))^{-k} .
	\end{equation}
\end{theorem}

An upper bound for the distance to $\BP$ was obtained in \cite{koldobsky2011isomorphic} (the power of the logarithmic term was later corrected in \cite[Proposition 1]{koldobsky2015slicing}).
\begin{theorem}
	\label{res_estimate_BP_upper}
	Let $K$ be an origin-symmetric convex body in $\R^n$, and let $1 \leq k \leq n-1$. Then
	\begin{equation}
		\label{eq_estimate_BP_upper}
	\dovr(K, \BP) \leq c \sqrt{n/k} \left(\log\left(\frac{en}k\right)\right)^{3/2}
	\end{equation}
	where $c>0$ is a universal constant.
\end{theorem}
Moreover, since $\BP \subseteq L_{-k}^n$ (this was proven in \cite{koldobsky2000functional} and \cite{milman2006generalized}, see also \cite[Theorem 4.23]{koldobsky2005fourier}), inequalities \eqref{eq_estimate_Lp_lower} and \eqref{eq_estimate_BP_upper} imply
\begin{equation}
	\label{eq_estimate_BP_lowerupper}
	c n^{\frac 1{2k}} \leq \sup_K \dovr(K, L_{-k}^n) \leq \sup_K \dovr(K, \BP) \leq c \sqrt{n/k} \left(\log\left(\frac{en}k\right)\right)^{3/2}.
\end{equation}
The two estimates are tight for large values of $k$ (at least proportional to $n$). For low values of $k>1$ the lower and upper bounds are far apart as $n \to \infty$. In this paper we use the same technique as in \cite{klartag2018example} to prove that Theorem \ref{res_slicing_k} is sharp up to a logarithmic term, thus obtaining the correct order (up to logarithms) of $\sup_K \dovr(K, \BP)$.
Our main theorem is the following:
\begin{theorem}
	\label{res_existsK}
	There exists a universal constant $c>0$ such that for every $n \geq 2$ and $1 \leq k \leq n-1$, there exists an origin-symmetric convex body $K \subseteq \R^n$ and an even continuous probability density $f$ in $K$ such that 
	for every subspace $F\subseteq \R^n$ of codimension $k$,
	\[\int_F f \leq \left( c \sqrt{\frac{ \log(n) k}n}\right)^k.\]
\end{theorem}

As a consequence of Theorems \ref{res_existsK} and \ref{res_slicing_k} we deduce that the estimate \eqref{eq_estimate_BP_upper} is also asymptotically sharp, up to logarithmic terms.
\begin{corollary}
	\label{res_estimate_BP_lower}
	There exists a universal constant $c>0$ such that for every $n \geq 2$ and $1 \leq k \leq n-1$, there exists an origin-symmetric convex body $K \subseteq \R^n$ such that 
	\[\dovr(K, \BP) \geq c \sqrt{n/k} (\log(n))^{-1/2}.\]
\end{corollary}

Combining Corollary \ref{res_estimate_BP_lower} and Theorem \ref{res_estimate_BP_upper} we obtain
\[c \sqrt{n/k} (\log(n))^{-1/2} \leq \sup_K \dovr(K, \BP) \leq c' \sqrt{n/k} \left(\log\left(\frac{en}k\right)\right)^{3/2}\]
where the supremum runs over all origin-symmetric convex bodies $K \subseteq \R^n$, and $c,c'>0$ are universal constants.

It should be noted that our proof is a simplification of those used in \cite{klartag2018example} and \cite{Klartag2020} since we do not intend to suppress any logarithmic term.
We thank prof. A. Koldobsky for suggesting this approach, and for many other useful discussions related to this article.

The author was supported by Grant RYC2021-031572-I, funded by the Ministry of Science and Innovation/State Research Agency/10.13039 /501100011033
and by the E.U. Next Generation EU/Recovery, Transformation and Resilience Plan.

\section{Proof of Theorem \ref{res_existsK} }

Theorem \ref{res_existsK}  shall be proven at the end of the section after some auxiliary results.
Throughout the section $c,c', c''$ will denote universal constants that may change from one line to the next one. We denote $\varphi(t) = e^{-t^2}$, the Euclidean unit ball of $\R^n$ is denoted by $\B$, the canonical vectors of $\R^n$ are $e_1, \ldots, e_n$ and $\Gamma$ is the gamma function.
For $F \in G_{n,n-k}$ and $x \in \R^n$, the Euclidean distance from $x$ to $F$ is denoted by $d(F,x)$.

First we recall two elementary inequalities from \cite{haddad2023radon} and \cite{klartag2018example}:
\begin{lemma}[{\cite[Lemma 2.4]{haddad2023radon}}]
	\label{res_estimate_gamma}
	For $0 \leq \mu \leq \lambda$ we have
	\[\lambda^\mu \Gamma(\lambda-\mu) \geq \Gamma(\lambda).\]
\end{lemma}

\begin{lemma}[{\cite[Lemma 3.1]{klartag2018example}}]
	\label{res_estimate_expectation_general}
	Let $Y_1, \ldots, Y_N$ be independent, identically distributed random variables attaining values in the interval $[0, 1]$. Let $p \in [0, 1]$ satisfy $p \geq \E Y_i$. Then,
	\[\P\left( \frac 1N \sum_{i=1}^N Y_i \geq 3p \right) \leq e^{-pN}.\]
\end{lemma}

The next elementary lemma concerns the parametrization of the sphere with respect to the so-called {\it bispherical coordinates}. It can be found in \cite[formula (3.9)]{markoe2006analytic}.
\begin{lemma}
	\label{res_changeofvariables}
	Consider the decomposition $\R^n = \R^m \times \R^k$ with $m+k = n$ and $m,k \in [1,n]$.
	If $f:S^{n-1} \to \R$ is a measurable function, then
	\begin{multline}
		\int_{S^{n-1}} f(v) dv = \int_{S^{m-1}} \int_{S^{k-1}} \int_0^{\pi/2} f(\cos(\alpha) x, \sin(\alpha) y) \\ \times\cos(\alpha)^{m-1} \sin(\alpha)^{k-1} d\alpha dy dx
	\end{multline}
	where integration on the sphere $S^{l}$ is with respect to the $l$-dimensional Hausdorff measure.
	The cases $k=1$ and $m=1$ are also included, the measure in $S^0=\{\pm 1\}$ being the counting measure.
\end{lemma}

\begin{lemma}
	\label{res_estimate_expectation}
	Let $n \geq 1$ and $\theta$ be a random point uniformly distributed in $S^{n-1}$.
	Let $F$ be any subspace of codimension $k$ with $1 \leq k \leq n$, then
	\[\E( \varphi(d(F, n \theta) ) ) \leq n^{-k/2}.\]
	(Notice that there is no constant in the right-hand side.)
\end{lemma}

\begin{proof}
	By rotational invariance, we may assume without loss of generality that $F$ is generated by $e_{k+1}, \ldots, e_n$.
	We decompose $\R^n = \R^k \times \R^{n-k}$ and notice that for every $(x,y) \in \R^k \times \R^{n-k}$ we have $d(F, (x,y)) = |x|$.

	To find the expectation we integrate, use Lemma \ref{res_changeofvariables} and then the change of variables $t= \cos(\alpha)$, $dt = - \sqrt{1-t^2} d\alpha$.
	\begin{align}
		\E( \varphi(d(F, n &\theta) ) )\\
		&= \frac 1{\kappa_n} \int_{S^{n-1}} \varphi(n d(F, \theta)) d\theta \\
		&= \frac 1{\kappa_n} \int_{S^{k-1}} \int_{S^{n-k-1}} \int_0^{\pi/2} \varphi(n \cos(\alpha)) \cos(\alpha)^{k-1} \sin(\alpha)^{n-k-1} d\alpha d x d y \\
		&= \frac {\kappa_k \kappa_{n-k}}{\kappa_n} \int_0^{\pi/2} \varphi(n \cos(\alpha)) \cos(\alpha)^{k-1} \sin(\alpha)^{n-k-1} d\alpha  \\
		&= \frac {\kappa_k \kappa_{n-k}}{\kappa_n} \int_0^1 \varphi(n t) t^{k-1} (1-t^2)^{\frac{n-k-2}2} d t,
	\end{align}
	where $\kappa_m =\frac 2{\Gamma(m/2)} \pi^{m/2}$ is the $m-1$ dimensional volume of $S^{m-1}$.

	To estimate the constant, we use Lemma \ref{res_estimate_gamma}.
	\begin{align}
		\label{eq_kappas}
		\frac {\kappa_k \kappa_{n-k}}{\kappa_n} 
		&=\frac{2 \Gamma(n/2)}{\Gamma(k/2) \Gamma \left(\frac{n-k}{2}\right)} \\
		&\leq 2 \Gamma(k/2)^{-1} \left(\frac n 2\right)^{k/2}.
	\end{align}

	For the integral we use a change of variables $s=nt$.
	\begin{align}
		\label{eq_integrals}
		\int_0^1 \varphi(n t) t^{k-1} (1-t^2)^{\frac{n-k-2}2} d t 
		&= n^{-k} \int_0^n \varphi(s) s^{k-1} \left(1-\left( \frac s n\right)^2\right)^{\frac{n-k-2}2} d s \\
		&\leq n^{-k}\int_0^n \varphi(s) s^{k-1} d s \\
		&= n^{-k} 2^{k/2-1} \Gamma(k/2)
	\end{align}

	Combining \eqref{eq_kappas} and \eqref{eq_integrals} we get the result.
\end{proof}

The Grassmanian $G_{n,n-k}$ of $k$-codimensional subspaces of $\R^n$ can be constructed as the quotient space of the group of rotations $O_n$, by the subgroup of rotations that fix a given $k$-codimensional subspace.
As such, $G_{n,n-k}$ inherits naturally a metric $d$, which is the quotient of the operator norm metric.
This is defined as
\begin{equation}
	\label{def_dG}
d_G(F_1, F_2) = \min\{\|I - T\|_{\rm op} : T \in O_n, T(F_1) = F_2\}
\end{equation}
for $F_1, F_2 \in G_{n,n-k}$.
Here $\|\cdot\|_{\rm op}$ denotes the operator norm, defined by 
\[\|S\|_{\rm op} = \sup_{\theta \in S^{n-1}} |S(\theta)|,\]
for every linear map $S:\R^n \to \R^n$.

First we need the following simple lemma.
\begin{lemma}
	\label{res_lipschitz}
	Let $\theta \in S^{n-1}$ be fixed, then the function $F \in G_{n,n-k} \mapsto d(F, \theta)\in\R$ is Lipschitz with constant $1$.
\end{lemma}
\begin{proof}
	Let $F_1, F_2 \in G_{n,n-k}$.
	By the definition of $d_G$, there exists $T \in O_n$ with $T(F_1) = F_2, \|T-I\|_{\rm op} = d_G(F_1, F_2)$.
	We have
	\begin{align}
	|d_G(F_1, \theta) - d_G(F_2, \theta)| 
		&= |d_G(F_2, T(\theta)) - d_G(F_2, \theta)| \\
		&\leq |T (\theta)- \theta| \\
		&\leq \|T-I\|_{\rm op} \\
		&\leq d_G(F_1, F_2),
	\end{align}
	and the lemma follows.
\end{proof}

For $\delta>0$, a $\delta$-net with respect to $d_G$ is a finite subset $\mathcal F \subseteq G_{n,n-k}$ such that for every $F \in G_{n,n-k}$ there is $F' \in \mathcal F$ with $d_G(F,F') < \delta$.
We shall use the following result by Szarek \cite{szarek1982nets}.
\begin{theorem}[{\cite[Proposition 8]{szarek1982nets}}]
	\label{res_netGrassmanian}
	For every $0 \leq \delta \leq \sqrt{2} = {\rm diam}(G_{n,n-k})$ there exists a $\delta$-net with respect to $d_G$ of cardinality less than $(c/\delta)^{k(n-k)}$, where $c>0$ is a universal constant.
\end{theorem}

\begin{proposition}
	\label{res_points}
	There exist a universal constant $c>0$ and $N = c n^{\frac k2 + 4}$ points $\theta_1, \ldots, \theta_N$ such that
	\[\frac 1N \sum_{i=1}^N \varphi( d(F, n \theta) ) \leq 4 n^{-k/2}\]
	for every $F \in G_{n,n-k}$.
\end{proposition}
\begin{proof}
	Let $\delta = n^{-k/2-1}$ and consider the corresponding $\delta$-net $\mathcal F$ of cardinality less than $(c' n^{k/2+1})^{k(n-k)} \leq e^{c'' n^4}$.
	Take $\theta_1, \ldots, \theta_N$ independent random points uniformly distributed in $S^{n-1}$.
	By Lemma \ref{res_estimate_expectation} we have $\E( \varphi( d(F, n \theta_i) ) ) \leq n^{-k/2}$.
	Applying Lemma \ref{res_estimate_expectation_general}, for every $F \in \mathcal F$ we have
	\begin{align}
		\P\left(\frac 1N \sum_{i=1}^N \varphi( d(F, n \theta_i) ) \geq 3 n^{-k/2}\right) 
		&\leq e^{-n^{-k/2} N}
		= e^{-cn^4}
	\end{align}
	\[\P\left(\exists F \in \mathcal F, \frac 1N \sum_{i=1}^N \varphi( d(F, n \theta_i) ) \geq 3 n^{-k/2}\right) \leq e^{c''n^4} e^{-c n^{4}}\]
	which is less than $1$ for all $1\leq k \leq n$, $n \geq 2$, for some universal constant $c>0$ depending on $c''$.
	We get the existence of $\theta_1, \ldots, \theta_N$ such that
	\[\frac 1N \sum_{i=1}^N \varphi( d(F, n \theta_i) ) \leq 3 n^{-k/2}\]
	for every $F \in \mathcal F$.

	By Lemma \ref{res_lipschitz}, the function $F \mapsto \frac 1N \sum_{i=1}^N \varphi( d(F, n \theta_i) )$ is Lipschitz with constant $n$.
	Then for any $F \in G_{n,n-k}$,
	\[\frac 1N \sum_{i=1}^N \varphi( d(F, n \theta_i) ) \leq n \delta +3 n^{-k/2} = 4 n^{-k/2}\]
	and the proposition follows.
\end{proof}

Let $\gamma_n(x) = (2\pi)^{-n/2} e^{-|x|^2/2}$ denote the standard Gaussian probability measure.
\begin{definition}
	\label{def_K_f}
	Let $\theta_1, \ldots, \theta_N$ be the points obtained in Proposition \ref{res_points}.
	Define the function $f_0(x) = \frac 1M \sum_{i=1}^M \gamma_n(x-p_i)$, where 
	\[\{p_1, \ldots, p_M\} = \{n\theta_1, -n \theta_1, n \theta_2, -n \theta_2, \ldots, n\theta_N, -n \theta_N\},\]
	and $M=2N$.

Define the convex set $K_0$ as the convex hull of the points $p_1, \ldots, p_M,$ and $\pm n e_1, \ldots, \pm n e_n$.
\end{definition}

The set $K_0$ and the function $f_0$ will satisfy the necessary properties of Theorem \ref{res_existsK}, after normalization and rescaling.
These properties will be stated separately in the following three propositions.

\begin{proposition}
	\label{res_prop_section}
	Let $f_0$ be as in Definition \ref{def_K_f}.
Let $F$ be any subspace of codimension $k$ with $1 \leq k \leq n-1$, then
	\[
		\int_{F} f_0 = (2\pi)^{-k/2} \frac 1M \sum_{i=1}^M \varphi(d(F, p_i)) \leq 4(2\pi n)^{-k/2}.
	\]
\end{proposition}
\begin{proof}
	Let $q_i$ be the orthogonal projection of $p_i$ onto $F$.
	Using the change of variables $y = x - q_i$,
	\begin{align}
		\int_F f_0
		&= \frac 1M \sum_{i=1}^M (2\pi)^{-n/2} \int_F e^{-\frac 12 |x - p_i|^2} d x \\
		&= \frac 1M \sum_{i=1}^M (2\pi)^{-n/2} \int_F e^{-\frac 12 |x - q_i|^2 - \frac 12 d(F, p_i)^2} d x \\
		&= (2\pi)^{-k/2} \frac 1M \sum_{i=1}^M  e^{- \frac 12 d(F, p_i)^2} (2\pi)^{-(n-k)/2} \int_F e^{-\frac 12 |y|^2} d y \\
		&= (2\pi)^{-k/2} \frac 1M \sum_{i=1}^M \varphi( d(F, p_i) ).
	\end{align}
	The inequality follows from Proposition \ref{res_points}.
\end{proof}

\begin{proposition}
	\label{res_prop_volume}
	The convex body $K_0$ in Definition \ref{def_K_f} satisfies
	\[\vol{K_0}^{1/n} \leq c \sqrt{\log(n) k},\]
	where $c>0$ is a universal constant.
\end{proposition}
\begin{proof}
	By a result of Gluskin \cite[Theorem 2]{gluskin1989extremal}, a polytope inside $\B$ which has less than $m$ vertices, has volume at most
	\[\left( c \frac {\sqrt{\log(m/n)}}n \right)^n,\]
	where $c>0$ is a universal constant.
	Since $K \subseteq n \B$,
	\begin{align}
		\vol{K}^{1/n} 
		&\leq c \sqrt{\log((M+2n)/n)} \\
		&\leq c \sqrt{\log((c n^{\frac k 2 + 4}+2n)/n)} \\
		&\leq c \sqrt{\log(n)(k/2 + 3)}
	\end{align}
	and the theorem follows with $c>0$ independent of $k$ and $n$, provided $n\geq 2$.
\end{proof}

\begin{proposition}
	\label{res_prop_measure}
	Let $K_0,f_0$ be as in Definition \ref{def_K_f}, then
	\[\int_{3 K_0} f_0  \geq 3/4.\]
\end{proposition}
\begin{proof}
	Since $K_0$ contains the vectors $\pm n e_i$, it contains $\sqrt n \B$.
	Then $3 K_0$ contains $K_0 + 2\sqrt{n}\B$ and
	\begin{align}
		\label{eq_estimate_measure_gamma}
		\int_{3 K_0} f_0  
		&\geq \int_{K_0 + 2 \sqrt n \B} f_0  \\
		&\geq \frac 1M \sum_{i=1}^M \int_{p_i + 2 \sqrt n \B} \gamma_n(x-\delta_{p_i}) dx  \\
		&= \int_{2 \sqrt n \B} \gamma_n(x) dx.
	\end{align}
	To bound the last integral from below we use the fact that $\int_{\R^n} |x|^2 d\gamma_n(x) = n$ and apply the Markov-Chebyshev inequality,
	\begin{align}
		\int_{|x|^2 \geq 4n} \gamma_n(x) dx
		&\leq \frac 1{4n} \int_{\R^n} |x|^2 d\gamma_n(x) = \frac 14,
	\end{align}
	and we obtain from \eqref{eq_estimate_measure_gamma},
	\[\int_{3 K_0} f_0 \geq \int_{2 \sqrt n \B} \gamma_n(x) dx = 1 - \int_{\R^n \setminus 2 \sqrt n \B} \gamma_n(x) dx \geq \frac 34.\]
\end{proof}

\begin{proof}[Proof of Theorem \ref{res_existsK}]
	Consider the set $K = K_0/\vol{K_0}^{1/n}$ and the function $f:\R^n \to \R$ defined by
	\[f(y) = \vol{3K_0} \left(\int_{3K_0} f_0\right)^{-1} f_0(\vol{3 K_0}^{1/n} y)\]
	for $y \in K$, and $f(0)=0$ otherwise.

	Then clearly $\vol{K}=1$ and $\int_K f = 1$, so $f$ is an even probability density in $K$, which is also continuous in $K$.

	By Propositions \ref{res_prop_section}, \ref{res_prop_volume} and \ref{res_prop_measure}, for any $F \in G_{n,n-k}$,
	\begin{align}
		\int_F f 
		&\leq \vol{3K_0} \left(\int_{3K_0} f_0\right)^{-1} \int_F f_0(3\vol{K_0}^{1/n} y) \\
		&= \vol{3K_0}^{k/n} \left(\int_{3K_0} f_0\right)^{-1} \int_F f_0\\
		&\leq 3^k \left( \sqrt{\log(n) k}\right)^k \frac 43 \times 4(2\pi n)^{-k/2} \\
		&\leq \left(c \sqrt{\frac{\log(n) k}n} \right)^k.
	\end{align}
\end{proof}

\bibliographystyle{abbrv}
\bibliography{references}

\end{document}